\documentclass[12pt]{amsart}
   \usepackage{graphicx}
  \usepackage{booktabs}
\usepackage{amscd,amsthm,amssymb,amsfonts,amsmath}

\usepackage{tikz}

\def\hcorrection #1 {\advance \hoffset by #1}
\def\vcorrection #1 {\advance \voffset by #1}
\vcorrection {-0.19in}
\hcorrection {-0.50in}

\begin{document}
\baselineskip 0.22in

\newcommand{\ds}{\displaystyle}
\newcommand{\be}{\begin{equation}}
\newcommand{\ee}{\end{equation}}
\newcommand{\bd}{\begin{description}}
\newcommand{\ed}{\end{description}}
\newcommand{\ba}{\begin{array}}
\newcommand{\ea}{\end{array}}
\newcommand{\ul}{\underline}
\newcommand{\ol}{\overline}
\newcommand{\ld}{\lambda}
\newcommand{\Dt}{\mbox{$\Delta$}}
\newcommand{\dt}{\mbox{$\delta$}}
\newcommand{\Og}{\Omega}
\newcommand{\bm}{\boldmath}
\newcommand{\eps}{\varepsilon}
\newcommand{\ra}{\rightarrow}
\newcommand{\dr}{\mbox{$\downarrow$}}
\newcommand{\ur}{\mbox{$\uparrow$}}
\newcommand{\rhpu}{\mbox{$\rightharpoonup$}}
\newcommand{\rhpd}{\mbox{$\rightharpoondown$}}
\newcommand{\ti}{\tilde}
\newcommand{\bi}{\begin{itemize}}
\newcommand{\ei}{\end{itemize}}
\newcommand{\df}{\ds\frac}
\newcommand{\lra}{\longrightarrow}
\newcommand{\al}{\alpha}
\newcommand{\Lra}{\Longrightarrow}
\newcommand{\p}{\partial}
\newcommand{\bt}{\begin{tabular}}
\newcommand{\et}{\end{tabular}}
\renewcommand{\arraystretch}{1.2} 

\newcommand{\imp}{\Longleftrightarrow}

\newtheorem{thm}{Theorem}[section]
\newtheorem{corollary}{Corollary}[section]
\newtheorem{lemma}{Lemma}[section]
\newtheorem{proposition}{Proposition}[section]
\theoremstyle{remark}
\newtheorem{definition}{Definition}[section]
\newtheorem{remark}{Remark}[section]
\newtheorem{example}{Example}[section]
\newtheorem*{notation}{Notation}
\numberwithin{equation}{section}

\newcommand{\B}{\mathcal{B} }
\newcommand{\bC}{{\mathbb C}}
\newcommand{\bZ}{{\mathbb Z}}
\newcommand{\bQ}{{\mathbb Q}}
\newcommand{\bR}{{\mathbb R}}
\newcommand{\bN}{{\mathbb N}}
\newcommand{\bT}{{\mathbb T}}
\newcommand{\bU}{{\mathbb U}} 
\newcommand{\bV}{{\mathbb V}}
\newcommand{\End}{\operatorname{End}}
\newcommand{\tr}{\operatorname{tr}}

\title{Lattice aggregations of boxes and symmetric functions}

\author{Natasha Rozhkovskaya}
\address{Department of Mathematics, Kansas State University, Manhattan, KS 66502, USA}
\email{rozhkovs@math.ksu.edu}

\thanks{}

\begin{abstract}
We introduce two lattice  growth models:   aggregation of $l$-dimensional boxes and  aggregation of  partitions with  $l$ parts. 
We describe  properties of the models:    the parameter set of aggregations,  the  moments of the random variable 
of the number of  growth directions,   asymptotical behavior  of  proportions of the most frequent transitions of   two- and three-dimensional 
self-aggregations.
\end{abstract}

\maketitle

\section{Introduction}

Growth processes, and, in particular, aggregations  of   randomly moving    particles and clusters
 are studied by  a wide variety of  theoretical models,    experiments, and numerical   simulations (see \cite{DF90, M98, M99, San2000,  SO86, WS81}  for examples related to the discussion here). Advances in computer  technology  and  fractal geometry  allowed  researchers to describe  many important  growth   parameters,  such as fractal dimension,  shapes, kinetics of formation.  These studies show that  with all the   simplicity of the growth rules,  the   result is   ``a devilishly difficult model to solve, even approximately'', as it is  emotionally expressed in  \cite{San2000}.

For this reason,    computational and experimental methods  remain as  the  main  investigation instruments 
of aggregation processes, while theoretical studies are restricted to  significant simplifications and assumptions on the  parameters. Yet, simplified analytical models   give valuable  insight  into  core characteristics of   growth process.

In this note we  introduce and study analytically a simplified growth model that is inspired  by   diffusion limited  cluster aggregation (DLCA) models. In such models, once two randomly moving  in a medium clusters attach to each other, they form   a new cluster  that  continues to move   and eventually finds other   clusters to  aggregate \cite{San2000,WS81}.

In our simplified model  the  $n$-dimensional clusters  are represented  by $n$-dimensional rectangular  boxes.  Substitution of an object by an axis-aligned minimal bounding box is commonly  used in approximation theory,  since it  leads to a less expensive, quick  and sufficient  evaluation of desired properties of the object. 
Two boxes    produce  on attachment   a new bigger  box by the procedure  described in Section \ref{Sec1}. 
The focus  of this study is on the local statistical  properties of  the  aggregates shapes, while  the fractal nature of aggregates is ``forgotten" in this model. This is not a conventional  approach to the DLCA investigation, since, as it is pointed out in \cite{San2000}, the  large-scale  fractal structures of a  cluster is dominated by non-local effects. Yet, we advocate here for the study of local growth  properties, since, as we will see below, they provide interesting  combinatorics.

 The paper is organized  as follows.
 A particular inspiration for  our considerations comes from the observations  in \cite{Sor1, Sor2}, where 
DLCA numerical simulations    describe the shapes of aggregates by inscribing them in the minimal  bounding boxes and measuring the ratios of the side lengths of these boxes. In Section \ref{Sec6} we recall the  interesting phenomena   on the proportions of  clusters that  was observed in these papers.

In Section \ref{Sec1} from the  definition of aggregation of boxes 
   we  deduce the  basic properties of the process and describe the parameter set of possible attachments.

 In  Section \ref{Sec2}  we describe  the random variable of the  number of directions in which an $n$-dimensional box grows under aggregation. We compute  probabilities and the generating function of the moments in terms of elementary symmetric functions.  In Section \ref{Sec3} we apply the results  of the previous section to the  aggregation of  a box with a  ``particle'' represented by a unit box.
  
  In Section \ref{Sec4} we introduce the second model of  aggregations of boxes up to equivalence of rotations, which   defines  families of transitional  probabilities  on   the set of partitions of $l$ parts.

   In Section \ref{Sec5} we come back to the observations outlined in Section \ref{Sec6} and describe the process of self-aggregation of partitions of  the highest transitional probabilities.  It turns out, that in two-dimensional case  such transitions  stabilize  to a Fibonacci type  sequence, which immediately implies the asymptotical limit  on the ratios of side lengths of the boxes. In three-dimensional case  a stabilizing sequence does not exist in general.

\subsection{Acknowledgement}  The author is very grateful to Prof.\,C.\,M.\,Sorensen for sharing   valuable insights  and providing  detailed  explanations on DLCA and RHA models. She would like to thank   Prof.\,A.\,Chakrabarti  for helpful discussions on the project. She also thanks   Jasmine Hunt,  for her  undergraduate research project  findings that made essential   contribution to    calculations  of Table  \ref{table:3}, and to  Abhinav Chand and James  Hyun, for writing a code that the author used to check her calculations in Table 10 and Figures 10-12. The research is supported by the the Simons Foundation Travel Support for Mathematicians grant MP-TSM-00002544.

\section{Motivations:  interpretations  of peak values  DLCA and RHM proportion distributions   }\label{Sec6}
  From the vast variety of  DLA  and DLCA  studies, we would like to  recall  the observations in  \cite{Sor1, Sor2} that   served as a special  motivation for introduced  here   models, and, in particular,  for the   study of    self-aggregations with the highest transitional probabilities  in Section \ref{Sec5}.

In
\cite{Sor1}    two computer  simulation models are investigated: the  diffusion limited cluster-cluster aggregation model (DLCA), and the  restricted hierarchical model  (RHM).
The later  one admits more theoretical  analysis  and  is used to  support an interesting interpretation of the maximal value point  of certain probability distributions of  the first model. 
Both models  measure   proportions of  rectangular boxes circumscribed about aggregated clusters.

The DLCA simulation starts with    $10^6$  monomers  placed  in a box.  At  each step, after  the number of clusters is counted, a random cluster is chosen and moved in a random direction at a distance of one monomer diameter. Probability of this movement  depends on the number of clusters and  the size of the chosen cluster. On the collision,  clusters irreversibly stick together, and the number of clusters goes down by one.
In addition to calculation of  fractal dimensions, the authors  
provide   distributions   of  shapes of  the $l$-dimensional   aggregates  ($l=2,3$) by comparing  their ``lengths''  in $l$ perpendicular directions. 

In  two-dimensional case, let  $L_1\ge L_2$ be the side lengths of a minimal bounding box of  a cluster. In  \cite{Sor1} 
  the distribution  of ratios  $L_1/L_2$ for the set of   the results of the DLCA simulation is computed. The computations show  that the  maximal value  of this distribution is achieved at   $L_1/L_2=1.63\pm 0.34$. Similarly, in  three-dimensional case,
let  $L_1\ge L_2\ge L_3$ be the side lengths of the  minimal bounding box of a cluster. According to \cite{Sor1},    distributions of ratios  $L_i/L_j$  ($i<j$) peaked at the values $L_1/L_2 =1.46\pm 0.27$,  $L_2/L_3 =1.35\pm 0.24$  and $L_1/L_3 =2.14\pm 0.27$. 

The authors suggest a very  interesting interpretation of  these experimental  values  as approximations of  irrational  constants,  $\varphi$,  $\psi$ and $\psi^2$,  where   $\varphi\simeq1.618$. is a root of $\varphi^2=\varphi+1$, 
and $\psi\simeq 1.465$ is a cubic root of $\psi^3=\psi^2+1$. This is a remarkable suggestion, since a natural presence of an  irrational number would indicate that the peak value point of the distribution  depends not so much  on a physical input, but  rather on mathematical properties of the model. 
  In  the support of this conjecture,  a  restricted hierarchical model (RHM)  was studied  in  \cite{Sor2}. This model is  more analytical and has   a significantly smaller number of possible aggregation results. It is based on the aggregation of cloned clusters  under the  ``side-to-end' condition: the longest side of the original cluster should be  necessarily linked with the shortest  side of its clone.
The side-to-end aggregations similarly appear as   T-model in  \cite{WB89}, where aggregating clusters are modeled by ellipses always sticking together  to create a T-shape, along with  the averaged model, where   non-zero probabilities  are limited only  to three types of  attachments.
In Section \ref{Sec5} we relate   interpretations  of the peak value  points of distributions of  proportions of clusters \cite{Sor1, Sor2} to  the most frequent  transitions of   self-aggregations of partitions.

\section{Aggregations of boxes: definitions and properties }\label{Sec1}
\subsection{Two growth models}

 \begin{figure}[h!]
\includegraphics [width=7cm] {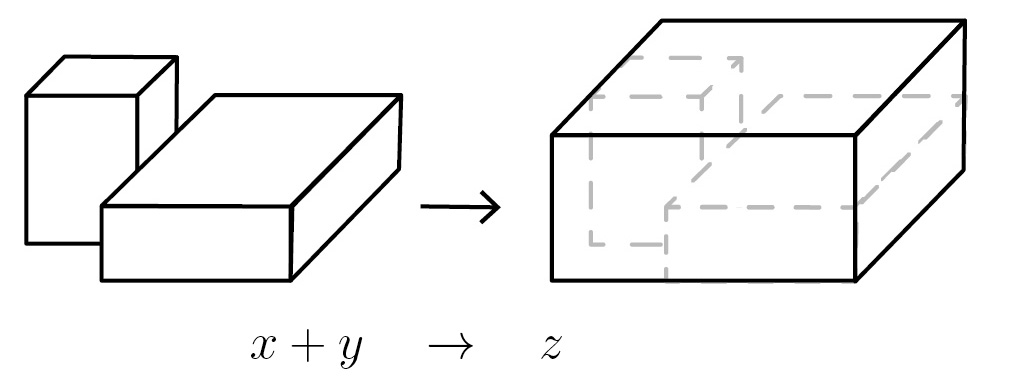}  
  \caption{Aggregation  of three-dimensional boxes. }
  \label{Fig2}
\end{figure}

In the growth models  of this note  an  $l$-dimensional  cluster   is represented by an $l$-dimensional  rectangular  box those movements are aligned with an integer lattice $\bZ^l$. Two boxes attach  to  each other along the lattice   to produce a new box (Figure \ref{Fig2}). For the purpose of  counting, there are two possible  variations (Figure \ref{Fig1}): 
\begin{enumerate}
\item 
 Different orientations of  a box are declared to  represent different clusters.  In this case each cluster  is identified with an  array of the side lengths of the box $x=(x_1,\dots, x_l)\in \bZ^l_{>0}$ written in a particular order. We refer  to this   growth model as  {\it aggregation of boxes}.

  \item 
  Alternatively, one can identify all rotations of a box to be  the same cluster. As it is explained in Section \ref{Sec4}, in this case  each cluster is identified with a partition $\lambda=(\lambda_1 \ge \dots \ge \lambda_l>0)$, $\lambda_i\in \bZ_{>0}$. This model is called    {\it aggregation of partitions}.
 
 \end{enumerate} 
 \begin{figure}[h!]
\includegraphics [width=10cm] {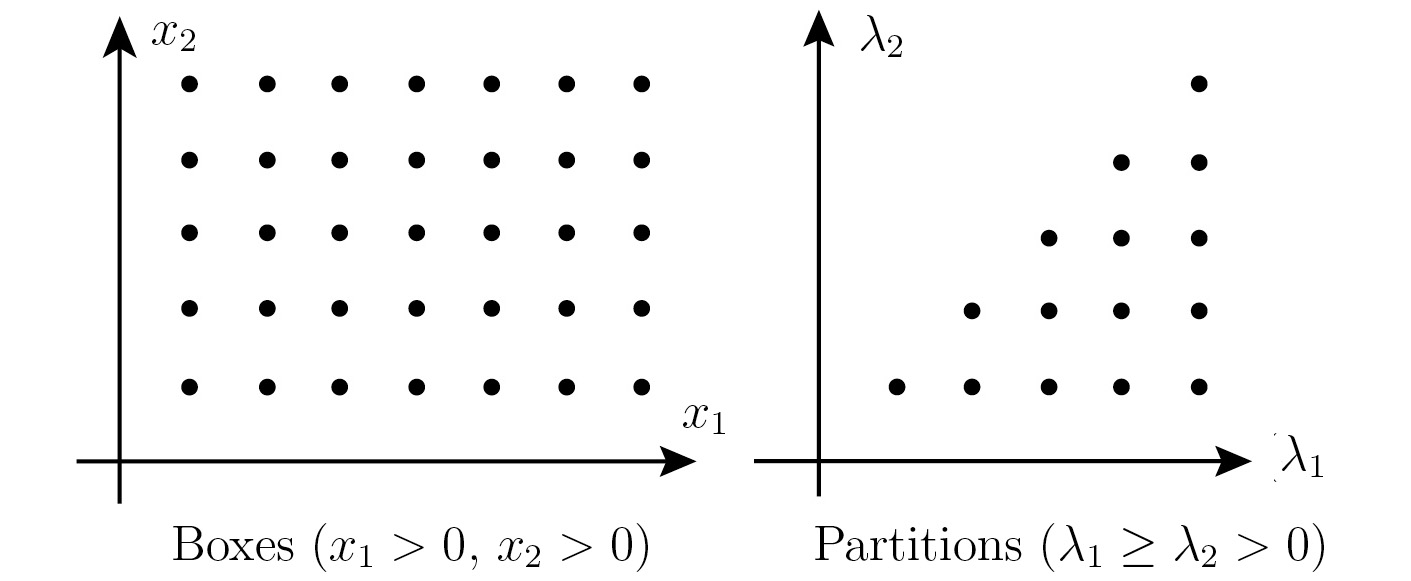}  \quad \quad 
  \caption{Enumeration of boxes and partitions in two-dimensional models.}
  \label{Fig1}
\end{figure}

In many applications  the assumption that  all rotations  represent  the same  shape  seem to be natural.  Combinatorics of  aggregations of partitions is more involved than combinatorics of aggregation of boxes,  but still it is   based on counting  of   contributing aggregations of  rotated boxes. Hence, we aim to study both  aggregation models.  
  Our goal is to describe  properties  of these two models and to compare    with     studies of  DLA  models in   \cite{Sor1}, \cite{Sor2}.

\subsection{Aggregation of boxes}

We fix  an  integer lattice $\bZ^l\subset \bR^l$ aligned with a   Cartesian coordinate  system. 
 Let the origin of the coordinate system 
coincide with  one of the vertices of the lattice.    
\begin{definition}
An { \it  $l$-dimensional  box}   $x$  is  a Cartesian product  of $l$  parallel  to  the  coordinate axes intervals with the integer  {lengths} represented by the array $  x=( x_1,\dots, x_l)
\in \bZ_{>0}^l$. Below we always assume that the endpoints of these intervals are  placed at the vertices of  the lattice.
\end{definition}   
Thus, we identify an $l$-dimensional  box with a point of  $\bZ_{>0}^l$.
For example, a rectangle of size $3$ by $5$ corresponds to  a point   $(3,5)$, and 
a rectangle of size $5$ by $3$  is a different point  with coordinates $(5,3)$.

Let 
 $x$  and  $y$ be  two $l$-dimensional boxes. The    box $y$  is translated without rotations, moving around  $x$, and  producing  a collection of   new boxes   by the following  attachment procedure (Figure \ref{Fig2}). 
\begin{itemize}
\item Boxes  $x$  and  $y$  are placed next to each other, so that  all  their sides are aligned with   the coordinate  axes of the $l$-dimensional space, and   vertices  of both boxes are  situated on  the lattice  $\bZ^l$.
\item  The boxes must be in the attachment: their boundaries must  contain  at least   one  common lattice  point. 
\item  The  union of the  two  attached boxes $x$ and $y$ is inscribed in a  minimal  bounding  box $z$ aligned with the  coordinate axes, as in Figure \ref{Fig2}.

\end{itemize}
\begin{definition}
We say that the box  $ z$ is
{\it a  result of  aggregation of two boxes $x$ and $y$},   and write `$x +  y\to z
$'.
\end{definition}
Various attachments of two boxes produce  a number of possible results, and the sizes of  resulting boxes  may  repeat for different attachments. Assuming that all attachments are equally likely, 
we introduce  the probability distribution  on all possible results   $P^{z}_{xy}=P(x+y\to z)$   as the number of occurrences of a  box  $ z$,  divided by the total number of ways for $x$ and $y$  to aggregate. 

\begin{example} \label{13+12}

There are 14  possible attachments of    $x=(1,3)$ and $y=(1,2)$,  and there are 14  possible attachments of     $x=(1,3)$ and $y=(2,1)$.  
The  probability distributions 
of aggregation results  are given in Tables  \ref{table:1}, \ref{table:2}.

\begin{table}[h!]
\centering

\begin{tabular}{ @{}cccccc @{} }
  \hline
 $z$& $(2,5)$& $(2,4 )$&$(2,3)$& $(1,5)$
 &otherwise\\
\hline
$P^{z}_{(1,3), (1,2)}$& $\frac{2}{7}$&  $\frac{2}{7}$ & $\frac{2}{7}$ & $\frac{1}{7}$& $0$\\
\hline 
 \end{tabular}
 
\caption{Results of aggregation $(1,3)+(1,2)$}
\label{table:1}
\end{table}

\begin{table}[h!]
\centering
\begin{tabular}{ @{}ccccc @{}  } 
  \hline
 $z $& $(3,4)$& $(3,3 )$&$(2,4)$ &otherwise\\
\hline
$P^z_{(1,3),(2,1)}$&  $\frac{2}{7}$ & $\frac{3}{7}$ &  $\frac{2}{7}$ & 0\\
\hline 
 \end{tabular}

\caption{Results of aggregation $(1,3)+(2,1)$}
\label{table:2}
\end{table}

\end{example}

\subsection{ Two-dimensional case: aggregation of rectangles} 
It is not difficult to compute   frequencies  of results of   aggregation of any  $2$-dimensional boxes   $x=(x_1,x_2)$ and $y=(y_1,y_2)$.
 The total number of possible attachments of these rectangles is 
 $|T|=2(x_1+x_2+ y_1+y_2)$ (see Proposition \ref{prop_parameter} for general formula). Probability distribution $P^z_{x,y}$ is provided in Table  \ref{table:3}.

\begin{table}[h!]
\centering

\begin{tabular}{@{}lc @{} } 
 \hline
 Result of aggregation  $ z=(z_1, z_2)$&Probability $P^z_{x,y}$   \\ 
\hline 
&\\
\includegraphics [height=1.05cm] {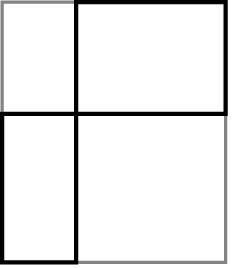}  \quad 
$(x_1+y_1, x_2+y_2),  $   &  $\frac{4}{|T|}$\\
\includegraphics [height=1.05cm]  {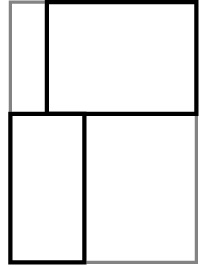}  \quad 
$(x_1+y_1-s, y_2+x_2), \quad s=1,\dots, \min(x_1, y_1)-1 $,   &  $\frac{4}{|T|}$\\
\includegraphics  [width=0.9cm] {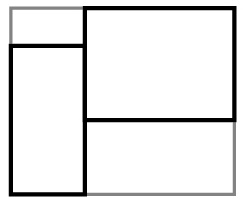}  \quad 
$ (x_1+y_1, x_2+y_2-s), \quad s=1,\dots, \min(x_2, y_2)-1$, &  $\frac{4}{|T|}$\\
\includegraphics  [width=0.9cm] {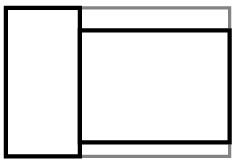}  \quad 
$(x_1+y_1, \max (x_2,y_2)) $  & $ \frac{2 |y_2-x_2|+2}{|T|}$ \\
\includegraphics  [height=1.05cm] {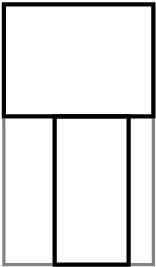}  \quad 
$( \max (x_1,y_1)\,, x_2+y_2) $  &  $ \frac{2|y_1-x_1|+2}{|T|}$ \\
\quad \quad \quad 
otherwise &0\\
\hline
  \end{tabular}
\caption{Probability distribution of aggregation of rectangles.}
\label{table:3}
\end{table}

\subsection{Basic properties of aggregation of boxes }

The following properties of probability distributions immediately follow from the definition. 
 \begin{proposition}\label{basic_prop}
Let  $x= (x_1,\dots, x_i, \dots, x_l)$ and  $y=(y_1,\dots, y_i, \dots y_l )$ be two $l$-dimensional boxes. 
\begin{enumerate}
\item 
$\sum_{z\in \bZ^l_{>0}}P^{z}_{x,y}=1$.
\item  $P^{z}_{x,y}=P^{z}_{y,x}$ {for any $ z\in \bZ^l_{>0}$.}
\item Let 
$x^\prime=(x_1,\dots, y_i, \dots, x_l)$ and  $ y^\prime= (y_1,\dots, x_i, \dots y_l )$  be  a pair of boxes  obtained  from the original pair    $x$ and $y$ by the
swap  of sides $x_i \leftrightarrow y_i$. Then both pairs of aggregating boxes  have the same sets of results  with the same probabilities: 
 \begin{align*}
P^{z}_{x,y}=P^{z}_{x^\prime ,y^\prime}
 \quad \text{for any $ z\in \bZ^l_{>0}$.}
 \end{align*}
 \item $P^{z}_{xy}\ne 0$ if  and only if $z=(z_1,\dots, z_l)$ is  within the range  $ max(x_i, y_i)\le z_i\le  x_i+ y_i $ for all $i\in \{1,\dots, l\}$.

 \end{enumerate}
\end{proposition}

 \subsection{Parameter set of aggregations}
In this section we describe the parameter set of possible attachments  of two $l$-dimensional boxes. 
This allows us to to compute the total number of aggregations for two boxes and compute some interesting probabilities.

\begin{proposition}\label{prop_parameter}

\begin{enumerate} 
Let $ x= (x_1,\dots, x_l)$ and   $ y= (y_1,\dots, y_l)$
 be two $l$-dimensional boxes. 

\item 
Let $R_{x+y}$ be an  $l$-dimensional box of size $(x_1+y_1)\times \dots \times  (x_l+y_l)$ fixed in the position  with one vertex in the origin: 
\[
R_{x+y}=\{(s_1,\dots s_l)| 0\le s_i\le x_i+y_i\}.
\]

The set of possible  attachments  of  boxes $x$ and $y$ in the process of their aggregation is  in one-to-one correspondence  with  the set  $ T$
 of all   integer points of the boundary  of $R_{x+y}$:
\[T=\bZ^l\cap \partial R_{x+y}\] 

\item 
Thus, number $|T|= |\{z| x+y\to z\}| $ of all  possible results of aggregations   $x+y\to z$ counted with multiplicities is given by 
\[
|T|= \prod_{i=1}^{l}(x_i+y_i+1)- \prod_{i=1}^{l}(x_i+y_i-1).
\]

\item \label{prop_p3}
Let $s=(s_1, s_2,\dots, s_l)\in T$ be one of the points of the parameter set producing  the result of aggregation
 $z(s)$. Then 
  $z(s)=(z_{x_1, y_1} (s_1),\dots z_{x_l, y_l} (s_l))$, where 
\begin{align}\label{eq_zs}
z_{x_i,y_i} (s_i)=\max(x_i,y_i,x_i+y_i-s_i,s_i), 
\end{align}
illustrated in Figure \ref{Fig3}.

 \begin{figure}[h!]
\includegraphics [width=7cm] {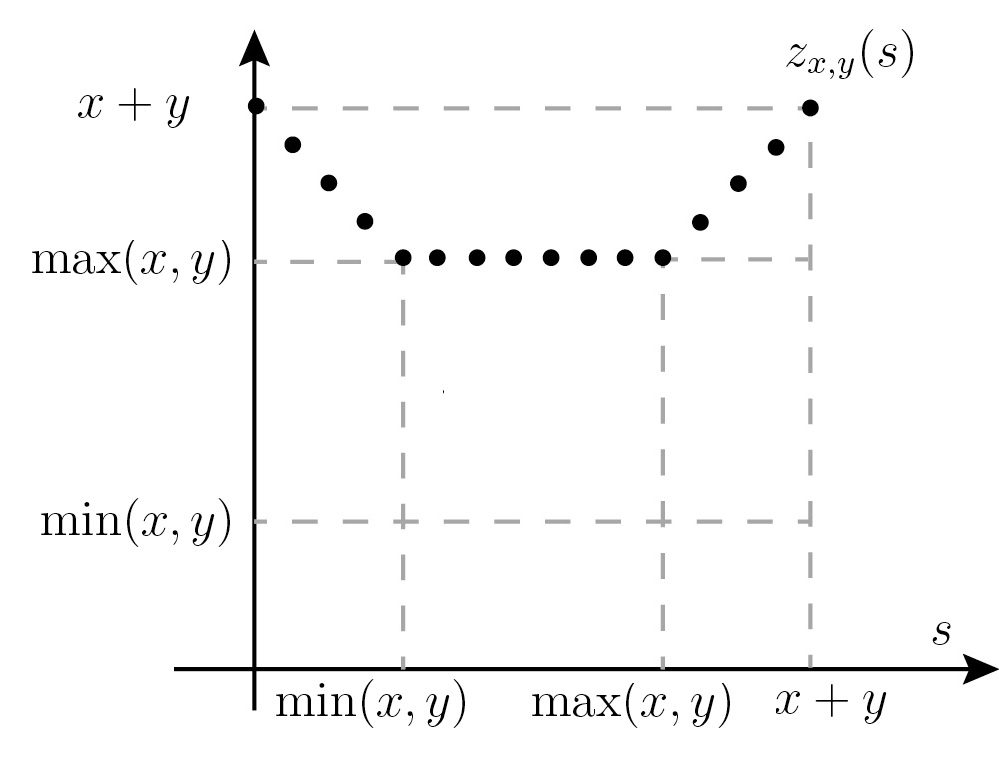}  
  \caption{Graph of $z_{xy}(s)$ of side lengths of results of aggregations.  }
  \label{Fig3}
\end{figure}

\end{enumerate}
\end{proposition}

\begin{proof}\begin{enumerate}
\item \label{proof32_1}
Recall that aggregations  $x+y\to z$ are produced by moving $y$ around $x$ in an  attachment. Each position of $y$ defines an aggregation result $z$, and at the same time, each position of  $y$ is completely defined by a position of any vertex of $y$. Mark a  vertex of $y$ and follow its movement  during  the  aggregation procedure. It is clear that  the position of this vertex   runs  through all
integer  boundary  points of  an $(x_1+y_1)\times \dots \times  (x_l+y_l)$ box.  Thus the set of attachments    of  $x$ and $y$  is  parametrized by  integer boundary points of such  box. We denote  this parameter box as $R_{x+y}$ and place  it  in the origin for  the convenience of  formulas and proofs  below.
 \item 
 The number $|T|$  of  the integer boundary  points     equals the difference between the   number of all  integer points   of   this  box
  $ \prod_{i=1}^{l}(x_i+y_i+1)$ and 
 the number of  interior integer points  of this box  $\prod_{i=1}^{l}(x_i+y_i-1)$. 

 \item 
 Let $x_i$ and $y_i $ be the $i$-th  sizes  of aggregating boxes.  From Proposition \ref{basic_prop} we can assume that $y_i\ge x_i$.
 Let   $A$ be  a marked vertex of the box as in the proof  of (\ref{proof32_1}). Possible values of coordinates $(s_1, \dots, s_l)$  of $A$ parametrize all possible attachments of $x$ and $y$.
 Since $y$ moves along each non-attaching coordinate independently, it is sufficient to prove (3) just  for the projection on one coordinate, see Figure 4.
As the box $y$ moves,  the coordinate $s_i$ runs through the values between $0$ and $x_i+y_i$. The corresponding size  $z_i(s_i)$ of the aggregation  result starts with the maximal possible value $x_i+y_i$ then  linearly drops down to the value  of $y_i$, remains on the same level until $s_i=y_i$, after that it again grows linearly until reaches the value 
  $z_i(x_i+y_i)= x_i+y_i$. The resulting  function for $z_i(s_i)$ is exactly of the form (\ref{eq_zs}).

 \begin{figure}[h!]
\includegraphics [width=8cm] {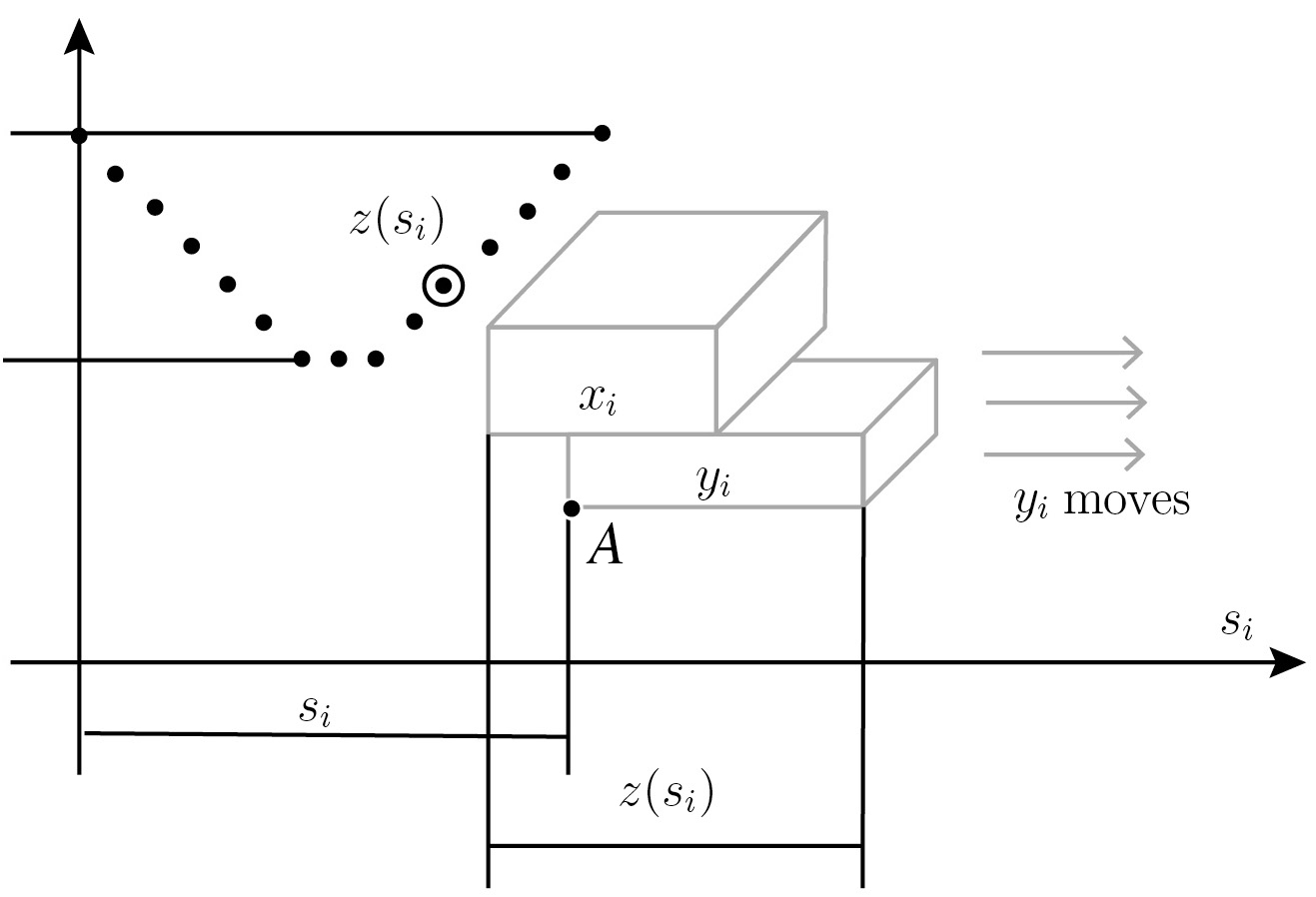}  
  \caption{ Varying  $s_i$ by moving  edge $y_i$ along  the edge $x_i$ and  measuring  the combined  size $z(s_i)$.   }
  \label{Fig4}
\end{figure}

\end{enumerate}
\end{proof}

\begin{example}
Figure \ref{Fig5} gives an example of the parameter set for the aggregations  $(1,1,4)+(5,2,3)$.
 \begin{figure}[h!]
\includegraphics [width=6cm] {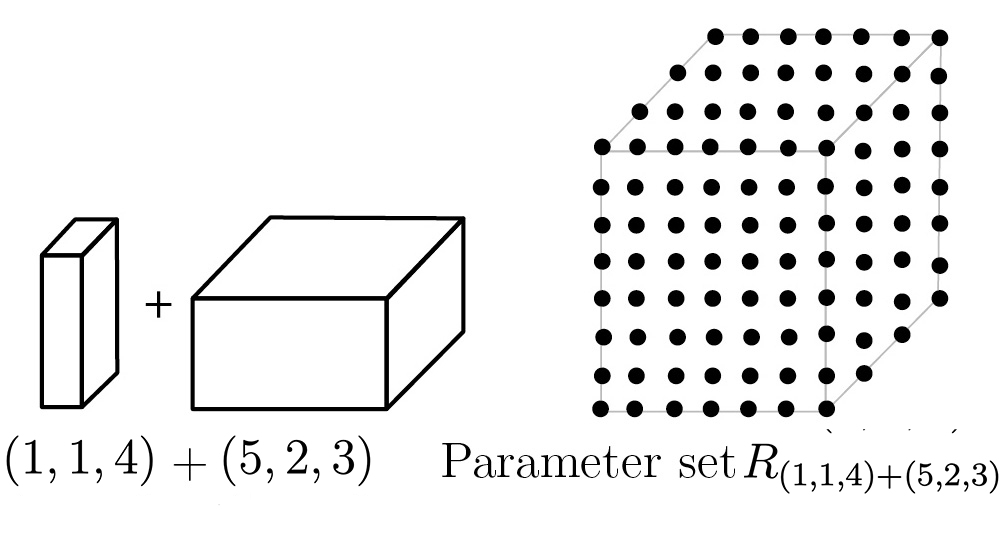}  
  \caption{Possible  attachments  of  $(1,1,4)$ and $(5,2,3)$ is parametrized by 
 integer points on the boundary of a box $R_{(1,1,4)+(5,2,3)}=(6,3,7)$. }
  \label{Fig5}
\end{figure}

\end{example}
\begin{example}
If $l=2$, then  the total number of aggregations with multiplicities is 
\[ |T|=2(x_1+x_2+ y_1+y_2)\]
If
$l=3$, then 
\[
|T|= 2+ 2(x_1+y_1)(x_2+y_2)+ 2(x_1+y_1)(x_3+y_3) +2(x_2+y_2)(x_3+y_3).
\] 
\end{example}

\section{The number of growth directions}\label{Sec2}
\subsection{Growth  of a box  in $k$ particular  directions}
 For the simplicity of notations, by Proposition \ref{basic_prop}, we  assume in this section that all $x_i \ge y_i$, $(i=1,\dots, l)$. Then  a side length $z_i$ of  a result of aggregation 
$x+y\to z$   either stays the same  $z_i = x_i$ or  grows $ z_i= x_i+t_i$,\, $0<t_i\le y_i$.  Our goal is to describe the random variable that  assigns  to each result of aggregation $z$ the number of sides that have  changed their lengths from the side  lengths of $x$.

\begin{proposition}
Let $z$ be a result of aggregation of $x+y\to z$,  where $x_i\ge y_i$ for all $i=1,\dots, l$. Let $ 1\le i_1<\dots <i_k\le l$, $1\le k\le l$.  Denote as  $P(i_1, \dots i_k)$ the probability that 
$z$  grows from $x$ in the directions $ \{i_1,\dots, i_k\}$:
\[
P(i_1, \dots, i_k)= P( z_j> x_j \,  \text{if and only if  }\, j\in \{i_1,\dots, i_k\}).
\]
Then
\begin{align}\label{prop_Pk}
P(i_1, \dots, i_k)=\frac{2^k}{|T|} \left(\prod_{j\in \{i_1,\dots, i_k\}} y_j- \prod_{j\in \{i_1,\dots, i_k\}} (y_j-1)\right)
\prod_{j\notin \{i_1,\dots, i_k\}}(x_j-y_j+1).
\end{align}
\end{proposition}
\begin{proof}
Without loss of generality we can assume that  $ \{i_1,\dots, i_k\}= \{1,\dots, k\}$. 
Each result  $z(s)$ that is different from $x$ in the first $k$ places is generated by  attachments that correspond to  parameters   $s_j\in [0, y_j-1]\cup[x_j+1, x_j+y_j]$ for $j=1,\dots, k$
and $s_j\in [y_j, x_j]$ for $j=k+1, \dots, l$, under the  condition that there is an  attachment  in  at least  one coordinate:  for at  least one $j\in\{1,\dots k\}$,  $s_j=0$ or $s_j=x_j+y_j$.

Assume for a moment  that for all  $j\in\{1,\dots, k\}$,    $s_j\in [0, y_j-1]$.   Using  the condition that  at least one of $s_j=0$  among these $j\in\{1,\dots k\}$, and  that 
$s_j\in [y_j, x_j]$ for $j=k+1, \dots, l$, we conclude that 
this set of parameters   is identified with a set of  integer points of  ``the exterior"  part of the boundary  the  $l$-dimensional box $[0, y_1-1]\times \dots \times [0, y_k-1]\times [y_{k+1},x_{k+1}]\dots \times [y_{l},x_{l}]$. Figure 6 illustrates this statement in three-dimensional case. The  integer points of the boundary of the  box $[0, x_1+y_1]\times[0, x_2+y_2]\times [0, x_3+y_3]$
represent the set $T$ of all parameters of attachments. The shaded parts of the boundary represent parameters $s$ that correspond to:
\begin{itemize}
\item    $k=3$, $0\le s_j<y_j$ for $j\in\{1,2,3\}$, and either $s_1=0$, or $s_2=0$, or $s_3=0$.
\item   $k=2$, $0\le s_j<y_j$ for $j\in\{1,2\}$, and either $s_1=0$, or $s_2=0$, and $y_3\le s_3\le x_3$.
\item   $k=1$, $s_1=0$, and   for $j\in\{2,3\}$ $y_j\le s_j\le x_j$.
\end{itemize}

 \begin{figure}[h!]
\includegraphics [width=11cm] {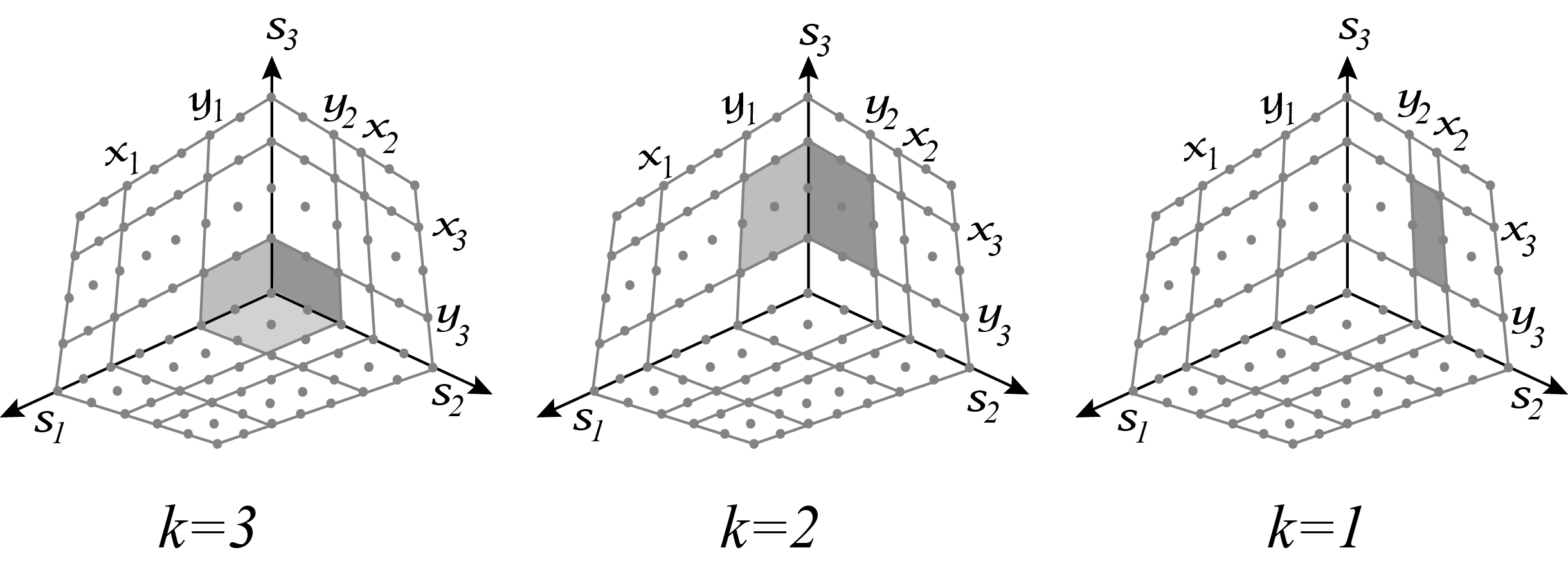}  
  \caption{ Parameters for changes in $k=3,2,1$ directions. }
  \label{Fig6}
\end{figure}

The  cardinality  of the described subset of parameters  is given by the difference of the number of integer points in the box
$[0, y_1-1]\times \dots \times [0, y_k-1]\times [y_{k+1},x_{k+1}]\dots \times [y_{l},x_{l}]$ and the box $[1, y_1-1]\times \dots \times [1, y_k-1]\times [y_{k+1},x_{k+1}]\dots \times [y_{l},x_{l}]$, which is 
\[
\prod_{j=k+1}^{l} (x_j-y_j+1)\left(\prod_{j=1}^ky_j-\prod_{j=1}^k(y_j-1) \right).
\]
Allowing each $s_j$, $j=1,\dots, k$  to vary through  two options of intervals,  $[0, y_j-1]$, or $[x_j+1, x_j+y_j]$, we get
$2^k$ copies  of the same subset to conclude that
\begin{align*}
P(1, \dots, k)=\frac{2^k}{|T|} \prod_{j=k+1}^{l} (x_j-y_j+1)\left(\prod_{j=1}^ky_j-\prod_{j=1}^k(y_j-1) \right).
\end{align*}
Statement (\ref{prop_Pk}) follows  by permuting the indices $\{1, \dots, k \}$ to get $\{i_1, \dots, i_k \}$.

\end{proof}
\subsection{Elementary symmetric functions}
Further combinatorics of this random variable is effectively described in the language of  symmetric polynomials \cite{Mac}.
\begin{definition} For $k=1,2,\dots $ the 
 {\it elementary symmetric}  polynomial $E_k (a_1, \dots,  a_n)$ in  $n$ variables $(a_1, \dots,  a_n)$, $n\ge k$,  is defined by
\[
E_k (a_1, \dots,  a_n)=\sum_{1\le i_1<\dots <i_k\le n} a_{i_1}\dots a_{i_k},
\]
and  $E_0=1$.
\end{definition}A well-known identity relates elementary symmetric functions and expansion of a polynomial with roots $a_1,\dots, a_n$:
\[
\prod _{{j=1}}^{n}(u -a_{j})=\sum_{k=0}^n (-1)^{n-k} E_{n-k}(a_1,\dots, a_n) u^k.
\]

For $1\le i_1< \dots <i_k\le l$, $1\le k\le l$, $ a, b\in \bR^l$, let  
\begin{align*}
r_{i_1,\dots, i_k}(a,b)=\prod_{j\notin \{i_1,\dots, i_k\}} a_j \prod_{j\in \{i_1,\dots, i_k\}} b_j,
\end{align*}
\[
R_k(a,b)= \sum_{1\le i_1<\dots <i_k\le l} r_{i_1,\dots, i_k}(a,b),
\quad 
 \mathcal R(a,b)[u]= \sum_{k=1}^{l}u^kR_k(a,b).
\]

\begin{lemma}\label{MC-1}
\begin{enumerate}
\item 
\begin{align*}
 \mathcal R(a,b)[u]= \prod_{i=1}^l(a_i+ub_i)
\end{align*}
\item $R_k(a,b)$ can be expressed through elementary symmetric functions: 
\begin{align*} 
R_k(a,b)= b_1\dots b_l E_{l-k}\left( \frac{a_1}{b_1},\dots, \frac{a_l}{b_l}\right)= a_1\dots a_l E_{k}\left(\frac{b_1}{a_1},\dots, \frac{b_l}{a_l}\right).
\end{align*}
\end{enumerate}
\end{lemma}
\begin{proof}
\begin{enumerate}
\item 
Observe that $u^kr_{i_1,\dots, i_k}(a,b)=r_{i_1,\dots, i_k}(a,ub)$, hence $u^kR_k(a,b)=R_k(a,ub)$, 
and it is enough to show that 
\begin{align*}
 \mathcal R(a,b)[1]=
\sum_{k=0}^{l} \sum_{1\le i_1<\dots <i_k\le l}  r_{i_1,\dots, i_k}(a,b) =\prod_{i=1}^l(a_i+b_i).
\end{align*}

Consider an $l$-dimensional box of sizes $(a_1+b_1,\dots,a_l+b_l)$. Clearly, its volume is $\prod_{i=1}^l(a_i+b_i)$. 
Divide each edge  of this box into  two parts,  of lengths $a_i$ and $b_i$, and cut the box into smaller boxes, as illustrated in the Figure \ref{Fig7} for $l=3$.
Then the total volume is   the sum of volumes of smaller boxes, which are  exactly of the form  $r_{i_1,\dots, i_k}(a,b)$.
 \begin{figure}[h!]
\includegraphics [width=3cm] {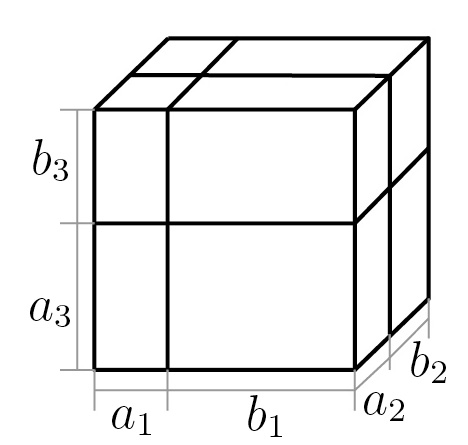}  
  \caption{ Cutting the  box by division of edges   into  sum of lengths $a_i$ and $b_i$, }
  \label{Fig7}
\end{figure}

\end{enumerate}
\end{proof}

\subsection{Random variable of the number of  directions growth}
Through this section  we again assume that  $x_i\ge y_i$ for all $i=1,\dots, l$.
Let  $X$   be  the random variable that assigns to a result of aggregation  $x+y\to z$,  the  number  of side lengths  of $z$  that are different from $x$. Then
\[
P(X= k)=\sum_{1\le i_1<\dots <i_k\le l} P(i_1,\dots, i_k),\]
 where 
in the introduced above notations 
\[
P(i_1, \dots, i_k)= \frac{2^k}{|T|} \left( r_{i_1,\dots, i_k}(x-y+1^l,y)- r_{i_1,\dots, i_k}(x-y+1^l, y-1^l)
\right)
\]
with $1^l= (1,\dots, 1) \in \bZ^l$.
Then  we get
\begin{proposition}
Let $x_i\ge y_i$ for all $i=1,\dots, l$. Then the probability that a result of aggregation  of $x$ and $y$  is  longer  than original box  $x$ in $k$ directions is given by 
\begin{align*}
P(X= k)
&=   \frac{2^k}{|T|} \left( R_k(x-y+1^l,y)- R_k(x-y+1^l, y-1^l)\right)
\\
&= \frac{2^k\prod_{i=1}^l(x_i-y_i+1)}{|T|} \left( E_k\left(\frac{y}{x-y+1^l}\right)- 
E_k\left(\frac{y-1^l}{x-y+1^l}\right)
\right),
\end{align*}
where we used short notations 
\[E_k\left(\frac{y}{x-y+1^l}\right)=E_k\left(\frac{y_1}{x_1-y_1+1},\dots ,\frac{y_l}{x_l-y_l+1}\right)\]
 and 
\[E_k\left(\frac{y-1^l}{x-y+1^l}\right)=
E_k\left(\frac{y_1-1}{x_1-y_1+1},\dots ,\frac{y_l-1}{x_l-y_l+1}\right).\]
\end{proposition}

The moments of the  random variable $X$ are defined in a standard way: 
\[
M^p=\sum_{k=1}^{l}  k^p \,P( X= k).
\]
 \begin{proposition}\label{moments}
   The moments $M^p$ of  $X$ can be computed by the formula
\begin{align*}
M^p&= \frac{ ( u\,\partial_u )^p\mathcal M(u))}  {\mathcal M(u) } |_{u=2},
\end{align*}
where 
\begin{align}\label{M(u)}
\mathcal M(u)
&= \prod_{i=1}^{l }\left(1+ \frac{y_iu}{x_i-y_i+1}\right) - \prod_{i=1}^{l }\left(1+ \frac{(y_i-1)u}{x_i-y_i+1}\right).
\end{align}

\end{proposition}
\begin{proof}
From the definition of moments we  deduce:
\begin{align*}
M^p&= \frac{1}{|T|}\sum_{k=1}^{l}   k^p\, {2^k}  R_k(x-y+1^l,y)-  \frac{1}{|T|}\sum_{k=1}^{l}   k^p\, {2^k}R_k(x-y+1^l, y-1^l)\\
&=\frac{1}{|T|}\left(\sum_{k=1}^{l}   k^p\, {u^{k}}  R_k(a,b)|_{a=x-y+1^l, b=y, u=2}- \sum_{k=1}^{l}   k^p\, {u^{k}}R_k(a,b)|_{a=x-y+1^l,b=y-1^l, u=2}\right).\\
\end{align*}
Let $ (k)_r= k(k-1)\dots (k-r+1)$. Recall that  $\partial_u^r(u^k)= (k)_r u^{k-r}$, $u^r\partial_u^r=(u\partial_u)_r$, and that 
$
\quad k^p=\sum_{r=0}^{p} S(p,r) (k)_r$, where  $S(p,r)$ are  Stirling numbers of the second  kind.
We get
\begin{align*}
\sum_{k=1}^{l} {k^p\,u^{k} R_{k}(a,b)}&= 
\sum_{k=1}^{l} \sum_{r=0}^{p} S(p,r) (k)_r\,u^{k} R_{k}(a,b) \\
&=\sum_{r=0}^{p} S(p,r) u^r \sum_{k=1}^{l} (k)_r\,u^{k-r} R_{k}(a,b)\\
&=\sum_{r=0}^{p} S(p,r) u^r\partial_u ^r\left(\sum_{k=1}^{l} u^k R_{k}(a,b)\right)\\
&= ( u\,\partial_u )^p\left(\sum_{k=1}^{l} u^k R_{k}(a,b))\right)=  ( u\,\partial_u )^p\mathcal R(a,b)[u].
\end{align*}
Then 
\begin{align}\label{Mp1}
M^p&= \frac{1}{|T|} ( u\,\partial_u )^p\left(\mathcal R(x-y+1^l,y)[u] -\mathcal R(x-y+1^l,y-1^l)[u]\right) |_{ u=2}.
\end{align}
Note that 
\begin{align*}
|T|=  {\prod_i(x_i+y_i+1) -\prod_i(x_i+y_i-1) }= \left(\mathcal R(x-y+1^l,y)[u] -\mathcal R(x-y+1^l,y-1^l)[u]\right) |_{ u=2}.
\end{align*}
To compete the proof  express everything in (\ref{Mp1}) through 

\begin{align*}
\mathcal M(u)&=\frac{1}{\prod (x_i-y_i+1)}\left(\mathcal R(x-y+1^l,y)[u] -\mathcal R(x-y+1^l,y-1^l)[u]\right)\\
&= \prod_{i=1}^{l }\left(1+ \frac{uy_i}{x_i-y_i+1}\right) - \prod_{i=1}^{l }\left(1+ \frac{u(y_i-1)}{x_i-y_i+1}\right).
\end{align*}
\end{proof}
\begin{corollary} \label{CorM}
Let $x_i\ge y_i$ for all $i=1,\dots, l$. On average a  result of aggregation  of $x$ and $y$  is longer than
 $x$  in  $M$ directions, where 
\[
M = 2 \ln (\mathcal M(u))^\prime |_{u=2}
\]
and  $\mathcal M(u)$ is given by (\ref{M(u)}).
\end{corollary}


\section{Aggregation  with a unit box}\label{Sec3}
\subsection{Properties of aggregation with a unit box.}
In case of aggregation with a  ``particle'', that is with a unit box  $y=(1^l)=(1,\dots, 1)$,    formulas of Section \ref{Sec2} simplify to nice expressions.  We summarize them in  Proposition \ref{MC-2}.  Clearly, each result  $z$ of aggregation $x+(1^l)\to z$ is an $l$-dimensional box $z=(z_1,\dots, z_l)$ those side lengths either remain  the same $z_i=x_i$, or grow by one, $z_i=x_i+1$. 

 \begin{proposition}\label{MC-2}
 In the  $l$-dimensional space consider  the process of the  aggregation of boxes $x+1^l\to z$.
 Then 

\begin{enumerate}
\item  The total number of ways for a particle $1^l$ to attach to the box $x$ is 
 \begin{align*}
 |T|= \prod_{i=1}^{l}(x_i+2)- \prod_{i=1}^{l}x_i
 \end{align*}
\item 
Each result of aggregation  $x+1^l\to z$  is of the form  $z=x+e_{i_1}+\dots + e_{i_k}$, where  $e_i=(0,\dots , \underset{i}{1}\dots, 0)$, $ 1\le i_1<\dots <i_k\le l$, and $1\le k\le l$. 
Transitional probabilities are given by 
\begin{align}\label{prop4_1}
P_{x, 1^l} ^{x+ e_{i_1}+ \dots +e_{i_k} } = P(i_1,\dots, i_k)=  \frac{2^{k} \, x_1\dots x_l}{x_{i_1}\dots x_{i_k}|T| }.
\end{align}
\item 
Probability  that  the resulting box  will grow  in size by $1$  in $k$ directions ($k=1,2,\dots, l$) is given by the formula
\begin{align}\label{prop4_2}
P( X=k)= \frac{2^{k} E_{l-k} (x_1,\dots, x_l)}{|T|}, 
\end{align}
where    $E_s(x_1,\dots, x_l)=\sum_{1\le i_1<\dots <\i_s\le l} x_{i_1}\dots x_{i_s}$ is the $s$-th elementary symmetric polynomial. 

\item On  average  through the aggregation with a particle  a box  $x$ grows in size in $M$  directions, where 
\[
M=\frac{\sum_{i=1}^l\frac{2}{2+x_i}}{1- \prod_{i=1}^l\frac{x_i}{x_i+2} },
\]

\end{enumerate}
\end{proposition}
\begin{proof}
The statements  follow by substitution $y=(1^l)$ in Propositions \ref{prop_parameter} nd \ref{MC-1}. For the last statement,  we compute the average number of directions as  the first moment of variable $X$ using Corollary \ref{CorM}:
 \begin{align*}
 M&= \frac{u\partial_u \left(\prod_{i=1}^l(1+\frac{u}{x_i}) - 1\right)}{\prod_{i=1}^l(1+\frac{u}{x_i}) - 1}|_{u=2}
 =\frac{\sum_{i=1}^l\frac{u}{u+x_i} \prod_{i=1}^l\left(1+\frac{u}{x_i}\right)}{\prod_{i=1}^l(1+\frac{u}{x_i}) - 1}|_{u=2}\\
  &=\frac{\sum_{i=1}^l\frac{u}{u+x_i}}{1- \prod_{i=1}^l\frac{x_i}{x_i+u} }|_{u=2}
    =\frac{\sum_{i=1}^l\frac{2}{2+x_i}}{1- \prod_{i=1}^l\frac{x_i}{x_i+2} },
 \end{align*}
 where we used (\cite{Mac} I, (2.10$^\prime$)): 
 \[
u \partial_u \prod_{i=1}^l\left(1+\frac{u}{a_i}\right)=\sum_{i=1}^l\frac{u}{u+a_i} \prod_{i=1}^l\left(1+\frac{u}{a_i}\right).
 \]
\end{proof}

\begin{example} $l=2$
\[M= \frac{x_1+x_2+4}{x_1+x_2+2}=1+ \frac{2}{x_1+x_2+2}.\]
Note that in this case  the average number of directions  depends on the perimeter of $x$, but does  not depend on  individual  lengths. As the perimeter grows to infinity, on average  the rectangle will  will grow  in one direction at each transition. Note also that $M<2$ for any rectangle  $x$.

\end{example}
\begin{example} $l=3$
\begin{align*}
M&= \frac{x_1x_2+x_1x_3+ x_2x_3+4(x_1+x_2+x_3)+12}{x_1x_2+x_1x_3+ x_2x_3+2(x_1+x_2+x_3)+4}\\
&=1+  \frac{2(x_1+x_2+x_3)+8}{x_1x_2+x_1x_3+ x_2x_3+2(x_1+x_2+x_3)+4}.
\end{align*}
If $x$ is not a unit cube$(1,1,1)$, then $x_1x_2+x_1x_3+ x_2x_3>4$, and  again we get  $M<2$.

\end{example}
\begin{corollary}
 Assume that  $x_i= O(N)$ as  $N\to \infty$ for all $i=1,\dots, l$.Then asymptotically 
 on   average  the size of the box $x$ grows in just  in  1  direction  by  the  aggregation with a unit box.
\end{corollary}
\begin{proof}
With   $x_i= O(N)$ as  $N\to \infty$ for all $i=1,\dots, l$ we can write that 
\begin{align*}
M=   \frac{2\sum_{i=1}^l\prod_{j\ne i}{(x_j+2)}}{\prod_{i=1}^l{(x_i+2)}- \prod_{i=1}^l{x_i} }
\to 1 \quad 
\text{
when $N\to \infty$.}
\end{align*}
\end{proof}

\subsection{Markov chain}\label{mark_1}
Repeating  aggregation  with a  unit box  $y= (1^l)$ 
one defines a Markov chain  on the   space  of  $l$-dimensional boxes  with  transition probabilities $\{p_{xz}= P_{x, (1^l)}^z\}$, given by (\ref{prop4_1}). Within standard definitions \cite{Chung},
the $n$-th step transition probabilities $\{p_{xz}^{(n)}\}$  are given  by recursion rule: 
\[
p_{x z}^{(0)}= \delta_{x,z}, \quad p_{xz}^{(1)}= p_{xz}=P_{x, (1^l)}^z, \quad p_{xz}^{(n)}=\sum_{t}p_{xt}^{(n-1)}p_{tz}. 
\]

\begin{example}
Let $l=2$. Then $x+(1,1) \to z$  produces three different results of aggregation $ \{ x+ e_1+e_2= (x_1+1, x_2+1), x+ e_1= (x_1+1, x_2), x+ e_2= (x_1, x_2+1)\}$. 
By (\ref{prop4_1})  transition probabilities are
\begin{align*}
p_{x, x+ e_1+e_2} = \frac{2}{x_1+x_2+2}, \quad 
p_{x, x+ e_1} = \frac{x_2}{x_1+x_2+2}, \quad 
p_{x, x+e_2} = \frac{x_1}{x_1+x_2+2}.
\end{align*}
One can write recurrence relation fo $n$-th transitional probabilities $p^{(n)}_{xz}$.  A box
 $z=(z_1,z_2)$ can be obtained as a result of aggregation from $(z_1-1,z_2-1)$, or $(z_1-1,z_2)$ or $(z_1,z_2-1)$.
Then
\begin{align*}
p^{(n)}_{x, (z_1, z_2)}=    \frac {2}{z_1+z_2}p^{(n-1)} _{x,(z_1-1, z_2-1)}+ \frac {z_2}{z_1+z_2+1} p^{(n-1)}_{x, (z_1-1, z_2)}+  \frac {z_1}{z_1+z_2+1}p^{(n-1)}_{ x,(z_1, z_2-1)}
\end{align*}
for any starting box $x$.
\end{example}

More  generally, one can formulate that in $l$-dimensional process, the $n$-th transitional probabilities satisfy the  recurrence relation
\begin{align}\label{trans}
p^{(n)}_{x, z}=  \sum_{k=1}^{l} \sum_{1\le i_1<\dots <i_k\le l} A_{i_1,\dots, i_k} p^{(n-1)}_{x, z- e_{i_1}-\dots - e_{i_k}},
\end{align}
where coefficients have a  quite involved form: 
\begin{align*}
A_{i_1,\dots, i_k}= \frac { 2^k\prod_{j\notin \{i_1\dots i_k\} } z_j}{ \prod_{j\in \{i_1\dots i_k\}}  (z_j+1) \prod_{j\notin \{i_1\dots i_k\} }( z_j+2) 
-\prod_{j\in \{i_1\dots i_k\}}  (z_j-1) \prod_{j\notin \{i_1\dots i_k\} } z_j    }
\end{align*}
Note  that (\ref{trans}) is a multidimensional analogue of Delannoy numbers recurrence relation  $D(n,k)=D(n-1,k)+ D(n-1,k-1) + D(n,k-1))$   with non-constant weights \cite{Del4, Del3, Del1, Del2}. It would be interesting to  investigate, if methods that are applied to  study Delannoy numbers sequence and its analogues could be  used  to 
understand the properties of  (\ref{trans}).

\section{ Aggregations of partitions}\label{Sec4}
\subsection{Two definitions }\label{Sec4-1}
Let us describe the growth model where all $90$-degree  rotations of  an $l$-dimensional box  around principle axis  represent the same cluster. 
\begin{definition}
  A  {\it partition} is a finite sequence  $\lambda=(\lambda_1, \dots, \lambda_l)$  of positive  integers  in non-increasing order: $\lambda_1\ge \dots\ge \lambda_l>0$, \cite{Mac}.
   The number   $l$  of coordinates is  called the {\it length of the partition}, and the sum of  the coordinates is denoted as $|\lambda|= \lambda_1+\dots+\lambda_l$. 
 \end{definition}

   For any $l$-dimensional  box  $x$ we  associate a partition $\lambda=(\lambda_1\ge \dots\ge \lambda_l)$ by rearranging the sizes of the box   $x$  in non-increasing order:
   \[
   \lambda_1=x_{\sigma(1)}\quad \ge\quad  \dots\ge \lambda_l=x_{\sigma(l)} \quad \text{for some}\quad \sigma \in S_l.
   \]
In this case  we write  $x\in [\lambda]$.
Note that  any permutation of coordinates of   $x$ can be matched with a   $90$-degree rotation of the box   around a  principle axis.
Indeed, a composition of  rotations of a box  $x$  permutes the lengths $ ( x_1,\dots, x_l)$, therefore,  any rotation of a box  is a box $(x_{\sigma(1)},\dots, x_{\sigma(l)})$ for some permutation $\sigma\in S_l$. A ninety-degree rotation around  a principle  axis  acts as  a transposition
$(x_1\dots x_i,\dots x_j,\dots  x_l)\mapsto (x_1,\dots x_j,\dots x_i,\dots  x_l) $. Symmetric group is generated by transpositions,  hence every possible permutation of components of  $x$ can be obtained by a composition of rotations. 

Thus,  $[\lambda]$ is exactly  the collection  of  all  different  $90$-degree  rotations around principle axis  of a box  with an array of sizes  $(\lambda_1,\dots, \lambda_l)$. Equivalence classes of  all $l$-dimensional boxes  up to their  $90$-degree  rotations are in one-to-one correspondence with the set of partitions of length $l$. 
\begin{example}
Let $\lambda=(3\ge 1\ge1)$. Then  $[\lambda]$ consists of three different boxes: 
$
[\lambda]=\{(3,1,1),\, (1,3,1),\, (1,1,3) \}. 
$
Let $\lambda= (a\ge\dots\ge a)$. Then  $[\lambda]$ consists of one  element, the  $l$-dimensional cube:
$
[\lambda]=\{(a,\dots, a) \}.
$
\end{example}

We define aggregation of partitions  in two equivalent ways. One of them is based on the  aggregation of  rotated versions of $l$-dimensional boxes, and the other one is based on permutations of parts of partitions.   
\begin{definition}\label{defp1}
 We say that a partition $\nu$  is a result of 
an  {\it aggregation of two partitions} $\lambda$ and $\mu$, and write 
$
[\lambda]+[\mu]\to [\nu],
$
 if  there exist  boxes $x\in[\lambda]$,  $y \in [\mu] $, and $z\in [\nu]$, such that 
$x+ y \to z
$.
\end{definition}
In other words, we fix an orientation  of the first box, and  consider  attachments of  all different orientations of the second  box. The  arrays of lengths of results 
of aggregations rearranged in non-increasing order  represent the  results of aggregations of partitions. 

Alternatively,  the same process of aggregation of partitions can be  defined  combinatorially.
\begin{definition}
Let $a, b \in \bZ_{>0}$. The {\it  set of overlaps of $a$ and $b$ } is the set of positive integers 
$\{c= a+b-s | s=0,1,\dots \min(a,b)\} $.

 If one thinks of $a$ and $b$  as lengths of two  integer strips, the set of overlaps consists of possible  lengths  covered by these  strips  where  one is placed on the top of the other,
 Figure \ref{Fig8}.
 
  \begin{figure}[h!]
\includegraphics [width=5cm] {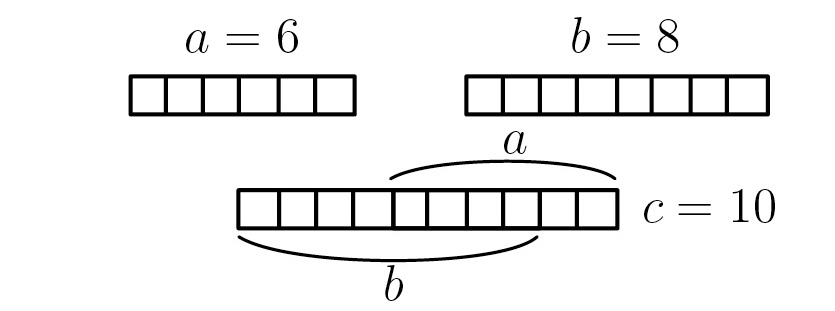}  
  \caption{  Example of  an overlap }
  \label{Fig8}
\end{figure}

We say that  the overlap  of the maximal  length $c= a+b$  is an {\it attachment}  of $a$ to $b$, and 
the  overlap  of the minimal  length $c=\max(a,b)$ is  an  {\it absorption} of  the smaller number from the pair $a,b $  by the  bigger one, Figure \ref{Fig9}.

  \begin{figure}[h!]
\includegraphics [width=5cm] {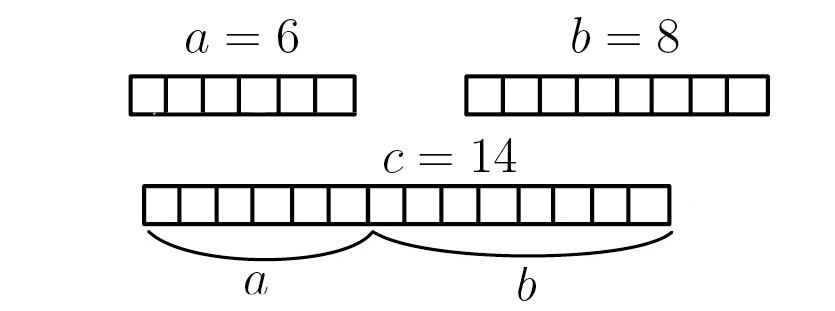}  \quad \includegraphics [width=5cm] {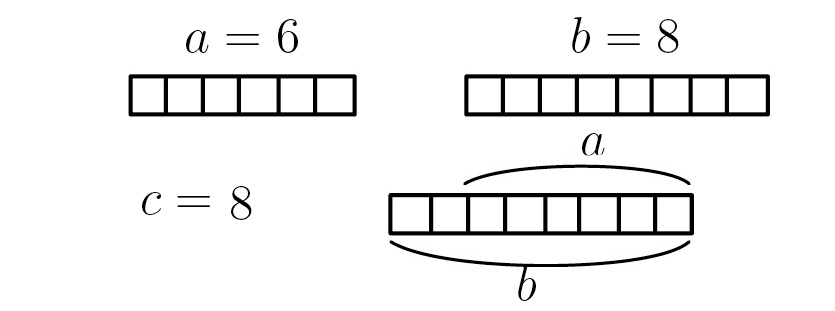}  
  \caption{ Examples of  an attachment and  an absorption }
  \label{Fig9}
\end{figure}

\end{definition}

\begin{definition} \label{defp2}
Let $\lambda= (\lambda_1\ge \dots \ge \lambda_l)$ and $\mu= (\mu_1\ge \dots \ge \mu_l)$
be two partitions  of length $l$. 
An aggregation of partitions $\lambda $ and $\mu$ is described by the following procedure.

\begin{enumerate}
\item  Fix a   permutation $\sigma$ of the set $ \{1,\dots, l\}$.
\item 
 For every $j \in  \{1,\dots, l\}$ create an overlap
  $\tilde \nu_j$  of parts $\lambda_j$ and $\mu_{\sigma(j)}$. One of  the overlaps
   $\tilde \nu_1,\dots, \tilde \nu_l $    must be necessarily an attachment. 
\item Put  parts   $\tilde \nu_1,\dots, \tilde \nu_l $   in non-increasing order   $\nu_1,\ge \dots \ge \nu_l$. 
 We say that the resulting partition $\nu$  is a {\it result of 
an   aggregation} of  $\lambda$ and $\mu$ and write \\
$
[\lambda]+[\mu]\to [\nu]
$.
\end{enumerate}
\end{definition}
From the discussion of Section \ref{Sec4-1} it is clear that Definitions \ref{defp1} and \ref{defp2}  are equivalent.

\subsection{Probability distributions of  aggregations of partitions.}\label{sec23}
Transitional probabilities  
$
P_{\lambda,\mu}^\nu= P( [\lambda]+[\mu]\to [\nu] )
$
of    partitions aggregations  can be computed from either definition.    Definition \ref{defp1} describes  $P_{\lambda,\mu}^\nu$ as the  number of  occurrences   of a  box  $z\in [\nu]$ divided by the total number of  all attachments  $x$ and $y$ for all  $y\in [\mu]$ and a fixed representative $x\in [\lambda]$.

Definition \ref{defp2}  equivalently  defines 
 $P_{\lambda,\mu}^\nu$ as  the  number of  occurrences  of partition $\nu$  created by implementing steps (1)-(3) of Definition \ref{defp2}
 for all permutations of parts of $[\mu]$ and for  all overlaps.

\begin{example}\label{ex4-1}
Tables \ref{table:4} --\ref{table:7} are examples of  probability distributions  of   partitions aggregations.
\begin{table}[h!]
\centering

{\small
\begin{tabular}{ @{}lccccc cc } 
  \hline
 $\nu $&[5,2]& [5,1]&[4,3]& [4,2]&[3,3]&[3,2]
 &otherwise\\
\hline
$P_{[3,1],[2,1]}^\nu$& $\frac{2}{14}$ &$\frac{1}{14}$& $\frac{2}{14}$ &$\frac{4}{14}$& $\frac{3}{14}$&$\frac{2}{14}$ &  0\\
\hline 
 \end{tabular}}
\caption{Results of aggregation $[3,1]+[2,1]$.}
\label{table:4}
\end{table}

\begin{table}[h!]
\centering

{\small
\begin{tabular}{@{} lcccccccc @{}  } 
  \hline
 $\nu $&[7,3]& [7,2] &[6,3]&[5,5]&[5,4]& [5,3]&[4,3]&otherwise\\
\hline
$P^{\nu}_{[4,2],[3,1]}$&  $\frac{1}{10}$&  $\frac{1}{10}$ &  $\frac{1}{10}$ &  $\frac{1}{10}$ &  $\frac{3}{10}$ &  $\frac{2}{10}$&  $\frac{1}{10}$ & 0\\
\hline 
 \end{tabular}}
\caption{Results of aggregation $[4,2]+[3,1]$.}
\label{table:5}
\end{table}

\begin{table}[h!]
\centering

{\small
\begin{tabular}{@{} lccccc cccccc @{}  } 
  \hline
 $\nu $&[8,3]&[8,2]& [7,3]&[6,5]& [6,4]&[6,3]&[5,5]&[5,4]&[5,3]&[4,3]
 &otherwise\\
\hline
$P_{[4,2],[4,1]}^\nu$& $\frac{2}{22}$& $\frac{2}{22}$& $\frac{2}{22}$ &  $\frac{2}{22}$ & $\frac{4}{22}$& $\frac{2}{22}$ & $\frac{2}{22}$ & $\frac{3}{22}$& $\frac{2}{22}$ &$\frac{1}{22}$ & 0\\
\hline 
 \end{tabular}}
\caption{Results of aggregation $[4,2]+[4,1]$.}
\label{table:6}
\end{table}

\quad\\
\begin{table}[h!]
\centering
{\small
\begin{tabular}{@{} lccccc ccc @{}  } 
  \hline
 $\nu $&[6,2,2]&[6,2,1]& [6,1,1]&[5,2,2]& [5,2,1]&[4,4,2]&[4,4,1] &
\\
\hline
$P_{[3,1,1],[3,1,1]}^\nu$& $\frac{8}{190}$&  $\frac{8}{190}$ & $\frac{2}{190}$ & $\frac{8}{190}$  & $\frac{8}{190}$ & $\frac{16}{190}$ & $\frac{8}{190}$ &\\
\hline 
 &&&&&&&\\
\hline 
 $\nu $&[4,3,2]&[4,3,1]& [4,2,2]&[4,2,1]& [3,3,2]&[3,2,2]&[3,2,1] & otherwise 
\\
\hline
$P_{[3,1,1],[3,1,1]}^\nu$& $\frac{48}{190}$&  $\frac{24}{190}$ & $\frac{8}{190}$ & $\frac{8}{190}$  & $\frac{36}{190}$ & $\frac{4}{190}$ & $\frac{4}{190}$ & 0\\
\hline 
 \end{tabular}}
 
\caption{Results of aggregation $[3,1,1]+[3,1,1]$.}
\label{table:7}
\end{table}

\end{example}

\subsection{Basic properties of  transitional probabilities }
Basic properties of   probability distribution  for  aggregation of partitions   follow from the definitions.
\begin{proposition}
 \begin{enumerate}
\item  
 For any partitions $\lambda, \mu$,
$
\sum_{\nu} P_{\lambda,\mu}^\nu =1
$
\item 
$P_{\lambda,\mu}^\nu=P_{\mu,\lambda}^\nu $  For any partitions $\lambda, \mu,\nu$,

\item 
$P^{\lambda,\mu}_\nu=0$ if $\nu_1\notin [\max(\mu_1, \lambda_1),  \mu_1+ \lambda_1]$.
\end{enumerate}
\end{proposition}

\section{Markov chains of self-aggregations }\label{Sec5}

\subsection{Repeated self-aggregations}
Similarly to Section \ref{mark_1}, one can consider a  Markov chain of  repeated aggregations with a fixed partition. 
In particular, since a unit box $y=(1^l)$ is invariant under all permutations of coordinates,  observations of  Section \ref{mark_1} can be 
translated from  the  model of   boxes  aggregation  to aggregation   of partitions in a straightforward way. 

In this section we study another class of  Markov chains,  based on self-aggregations of  partitions. This process can be  related to monodisperse or hierarchical aggregations, see e.g. \cite{BJK84}, \cite{WB89}.
Self-aggregation is an  aggregation where  a partition aggregates with its own copy.  The transition probabilities   on the space of 
partitions of length $l$  are $\{\tilde p_{\lambda,\nu}=P_{\lambda,\lambda}^\nu\}$.
Then the $n$-th step transition probabilities $\{\tilde p_{\lambda, \nu}^{(n)}\}$ are defined by recursion rule: 
\[
\tilde p_{\lambda, \nu}^{(0)}= \delta_{\lambda,\nu}, \quad \tilde p_{\lambda, \nu}^{(1)}= P_{\lambda, \lambda}^{ \nu},
\quad \tilde p_{\lambda,\nu}^{(n)}=\sum_{\mu} \tilde  p_{\lambda\mu} ^{(n-1)}\tilde p_{\mu,\nu}^{(1)}.
\]
Constructions below are motivated by observations of \cite{Sor1}, \cite{ Sor2} outlined in Section \ref{Sec6}. 
We aim to trace analytically the presence of mathematical constants $\varphi$ and $\psi$ in two- and three- dimensional growth process (see Section  \ref{Sec6}).
  We follow the path of the  most frequent  transitions of self-aggregations and look at the asymptotical behavior  of the proportions of boxes that appear in  such sequence.  In two-dimensional case 
  our results do  involve $\varphi$ in the description of the  limiting proportions of the  most frequent  transitions of self-aggregations. In three-dimensional case our methods  do not 
 trace  the presence of  $\psi$, but  further study would be  necessary to understand better the behavior  of the  most frequent transitions in  self-aggregation growth.

  \subsection{The highest  probability weight transitions}\label{Sec_7.3}
  
  \begin{definition} 
We say that  partition  $\nu $   is the  aggregation result   of  two partitions $ \lambda$ and $\mu$    with  {\it the highest probability weight} or is {\it the most frequent transition},  if the  probability 
$P_{\lambda,\mu}^\nu$ has the maximal value over  the set of all possible results of aggregations $[\lambda] +[\mu] \to [\nu]$.

\end{definition}
Table \ref{table:11} provides  the highest probability weight   aggregation results based on Example \ref{ex4-1}.
\begin{table}[h!]
\centering

\begin{tabular}{ @{}lc @{} } 
 \hline
Aggregation & The highest probability weight result \\ 
\hline 
 $[3,1]+ [2,1]$  &  $[4,2]$  \\
 $[4,2]+ [4,1]$ &  $[6,4]$\\
 $[4,2]+ [3,1]$ &   $[5,4]$ \\
$[3,1,1]+[3,1,1]$  & $[4,3,2]$  \\
  $[\lambda]+[1^l]$  & $(\lambda_1\ge  \dots \ge \lambda_{l-1}\ge \lambda_l+1>0)$\\
\hline
  \end{tabular}
\caption{Examples of the highest probability weight  results of  partitions aggregations }
\label{table:11}
\end{table}
Next we  describe the growth along the most frequent transitions in two-dimensional case. 
\begin{proposition} \label{prop62d}
\quad 
\begin{enumerate}
\item 
Consider  self-aggregation  $[\lambda] +[\lambda]\to [\nu]$ of a partition $\lambda=(\lambda_1, \lambda_2)$,  where $\lambda_1>\lambda_2+3$. Then the highest  probability weight   result  is  $\nu=(\lambda_1+\lambda_2, \lambda_1)$.
\item
Let $\lambda^{(1)}=(\lambda_1, \lambda_2)$ with $\lambda_1>\lambda_2+3$. Consider  a sequence  of self-aggregations of partitions
\[
[\lambda^{(N)}]+[\lambda^{(N)}]\to [\lambda^{(N+1)}], \quad  N=1,2,\dots, \quad 
\]
 where on each step  the highest probability weight result is picked as   $\lambda^{(N+1)} $   for the next  step of aggregation.
Then the shape of the rectangle $\lambda^{(N)}$ tends to the  ``golden rectangle":
\[
\lim_{N\to\infty }\frac{\lambda^{(N)}_1}{\lambda^{(N)}_2}= \varphi =\frac{1+\sqrt 5}{2} \simeq 1.6180, 
\]
and transitional probabilities of the sequence have the limit 
\[
\lim_{N\to\infty }P^{\lambda^{(N)}\, \lambda^{(N)}}_{\lambda^{(N+1)}}= \frac{1}{2}-\frac{1}{\varphi+1}\simeq 0.1180,
\]

\end{enumerate}
\end{proposition} 
\begin{proof}
\begin{enumerate}
\item 
Let  $\nu$ be a result of aggregation $[\lambda]\, +\,[\lambda]\,\to \, [\nu]$.
Any  rectangle $z \in [\nu]$  is either a result of  aggregation  $(\lambda_1, \lambda_2)+ (\lambda_1, \lambda_2)\to z$, or 
$(\lambda_1, \lambda_2)+ (\lambda_2, \lambda_1)\to z $.  From Table \ref{table:3}
we get the transitional probabilities for  self-aggregation of partitions of two parts, Table \ref{table:8}.
\begin{table}[h!]
\centering

 \begin{tabular}{@{} lcc@{} } 
 \hline
$\nu$&$P_{\lambda,\lambda}^\nu$  \\ 
 \hline
 $(2\lambda_1,2\lambda_2)$&$\frac{1}{2(\lambda_1+\lambda_2)}$\\ 
  $(2\lambda_1-s,2 \lambda_2), \quad s=1,\dots,\lambda_1 -1 $   &   $\frac{1}{2(\lambda_1+\lambda_2)}$\\ 
 $(2\lambda_1,2 \lambda_2-s), \quad s=1,\dots,\lambda_2 -1 $   & $\frac{1}{2(\lambda_1+\lambda_2)}$\\
$(2\lambda_1,\lambda_2)$&$\frac{1}{4(\lambda_1+\lambda_2)}$\\
$(\lambda_1,2\lambda_2)$&$\frac{1}{4(\lambda_1+\lambda_2)}$\\
 $(\lambda_1+\lambda_2,\lambda_1+\lambda_2)$&$\frac{1}{2(\lambda_1+\lambda_2)}$\\
  $(\lambda_1+\lambda_2,\lambda_1+ \lambda_2-s), \quad s=1,\dots,\lambda_2 -1 $   &$\frac{1}{\lambda_1+\lambda_2}$\\
  $(\lambda_1+\lambda_2,\lambda_1)$ &$\frac{\lambda_1-\lambda_2+1}{2(\lambda_1+\lambda_2)}$\\ 
  otherwise & $0$\\
 \hline 
 \end{tabular}
 \caption{Probability distribution for self-aggregation of $[\lambda_1, \lambda_2]$. }
\label{table:8}
\end{table}

It is clear that if $\lambda_1-\lambda_2>3$, the  frequency  $P_{\lambda, \lambda}^{(\lambda_1+\lambda_2,\lambda_1)}=(\lambda_1-\lambda_2+1)/2(\lambda_1+\lambda_2)$ of the  aggregation result
$\nu =(\lambda_1+\lambda_2,\lambda_1)$ is not only the highest value in the table, but  it is  higher than the sum of  frequencies  of any other three entries. Hence, even if different values of $s$ produce results of the same sizes,  $\nu$ remains to be the most frequent one. 

\item 
Note that ${\lambda^{(N)}_1}-{\lambda^{(N)}_2}=\lambda^{(N-1)}_2>3$ for every $N=2,3,\dots$. Hence,  from part (1),  by induction on the number of the step, 
the $N$-th result of the sequence of the most frequent  result of self-aggregations  has   parts  
\[
{\lambda^{(N)}_1}= F_N\lambda_1+ F_{N-1}\lambda_2 , \quad {\lambda^{(N)}_2}=F_{N-1}\lambda_1+ F_{N-2}\lambda_2,
\]
where  $F_1=F_2=1$, $F_{k+1}= F_{k-1}+F_{k}$ is the Fibonacci sequence.
This  implies the  limit of ratios of sizes and of transitional probabilities:  
\[
\lim_{N\to\infty }\frac{\lambda^{(N)}_1}{\lambda^{(N)}_2}=\lim_{N\to\infty }\frac{F_N\lambda_1+ F_{N-1}\lambda_2}{F_{N-1}\lambda_1+ F_{N-2}\lambda_2}=  \varphi, 
\]

\[
\lim_{N\to\infty }P_{\lambda^{(N)}\, \lambda^{(N)}}^{\lambda^{(N+1)}}=\lim_{N\to\infty }  \frac{ \lambda^{(N)}_1/\lambda^{(N)}_2-1+1/\lambda^{(N)}_2}{  2(\lambda^{(N)}_1/\lambda^{(N)}_2+1)}=
\frac{\varphi-1}{2(\varphi+1)}= \frac{1}{2}-\frac{1}{\varphi+1}.
\]

\end{enumerate}
\end{proof}
\begin{remark}
 From Table \ref{table:8}  observe that  $\lim_{N\to\infty }P_{\lambda^{(N)}\, \lambda^{(N)}}^{\nu}=0$  for all other $\nu\ne\lambda^{(N+1)}$. \end{remark}
 
\begin{example}  We illustrate  Proposition \ref{prop62d} with  Table \ref{table:9} of computer-aided calculations of a  sequence self-aggregations  that starts with $\lambda^{(1)}=[10,6]$
and picks for the next step the highest probability weight result.
\begin{table}[h!]
\centering

{\small
\begin{tabular}{@{} clll } 
  \hline
  $N$	 &  	$\lambda^{(N)}$ 	&	 $\lambda_1^{(N)}/\lambda_2^{(N)} $ 	&  	 $P_{\lambda^{(N-1)}\, \lambda^{(N-1)}}^{\lambda^{(N)}}$
  \\
\hline 
1   & [10, 6]	&1.667	& $\frac{6}{40}=0.1500$	\\
\hline 
2   & [16, 10]	&1.6		& $\frac{10}{64}\simeq0.1563$ 	\\
\hline 
3 & [26, 16]	& 1.625	& $\frac{14}{104}\simeq 0.1346$	\\
\hline 
4  & [42, 26]	& 1.6154\	& $\frac{22}{168} \simeq 0.1310$	\\
\hline 
5   &[68, 42] &1.6190	&  $\frac{34}{272} =0.1250$	\\
\hline 
6   & [110, 68]	&1.6176	&  $\frac{54}{440}\simeq 0.1227$ 	\\
\hline 
7   &[178, 110]&1.6182	& $\frac{86}{712}\simeq 0.1208$ 	\\
\hline 
8   &[288, 178] & 1.6180	&$\frac{138}{1152}\simeq 0.1198$	\\
\hline 
9   &[466, 288]&1.6181	& $\frac{222}{1864}\simeq 0.1191$	\\
\hline 
 \end{tabular}}
 \caption{A sequence of highest probability weight  self-aggregations.}
\label{table:9}
\end{table}

\end{example}

\subsection{The most  frequent  self-aggregation of three-dimensional  partitions }
The direct analogy of the results of Section \ref{Sec_7.3} in the three-dimensional case
that would support interpretations of of \cite{Sor1}, \cite{ Sor2} 
 would be a stabilization of (almost) any sequence of the most frequent  aggregations to a pattern  of the form  $[\lambda_1, \lambda_2, \lambda_3] \to [\lambda_1+\lambda_3, \lambda_1, \lambda_2]$. This would  lead to  natural  appearance of $\psi$  in the  asymptotical  description  of proportions of    highest probability weight self-aggregations. 
However, we did not observe such phenomena in general.  It seems that in many typical examples after a few steps the process    produces   not just one, but  a family of  most frequent  results with the same probability weight, as illustrated in Figures \,\ref{Fig11}-\ref{Fig13}. Moreover, once this process hits a cube $(a,a,a)$ shape,  the next step  brings   a very  large  number of  equally likely  most frequent  results (Figures \,\ref{Fig11} and \,\ref{Fig13}).

  Thus,  tracing the growth of   highest probability weight self-aggregations  provides an evidence for the proposed in \cite{Sor1, Sor2}  interpretation of the peak of probability distribution of proportions   for $l=2$,  but does not detect any  role of $\psi$   in  the description of the   peak of probability distribution of proportions   for $l=3$. We hope to investigate in more details  the  typical behavior of  the most frequent transitions of  self-aggregations   for $l\ge3$  elsewhere.
           \begin{figure}[h!]
\centering
\begin{tikzpicture}[node distance= 2.1cm]
 \tikzstyle{every node}=[font=\tiny]
\node(A)                           {$[10,3,1]$};
\node(B1)      [right of=A]       {$[10,10,4]$ };
\node(B2)      [right of=B1]       {$[20,10,10]$ };
\node(C1)      [right  of=B2]       {$[20,20,20]$ };
\node(D3)      [ right of=C1]   {($171$ result) };
\draw(A)       edge [ above]  node {0.1225} (B1);
\draw(B1)      edge [ above]  node {0.0438} (B2);
\draw(B2)      edge [ above]  node {0.0394} (C1);
\draw(C1)       -- (D3);  
\end{tikzpicture}
\caption{ A sequence of highest probability weight self-aggregation  results that starts with   $\lambda=[10,3,1]$.}
\label{Fig11}
\end{figure}

          \begin{figure}[h!]%
\centering
\begin{tikzpicture}[node distance= 2.3cm]
 \tikzstyle{every node}=[font=\tiny]
\node(A)                           {$[3,1,1]$};
\node(B1)      [right of=A]       {$[4,3,2]$ };
\node(C1)      [right  of=B1]       {$[6,5,4]$ };
\node(D1)      [below right of=C1]  {$[10,9,6]$};
\node(D2)      [above right of=C1]  {$[10,9,8]$};
\node(E1)      [below right of=D2]  {$[18,17,16]$};
\node(F1)      [right of=E1]  {$\dots$};
\node(E2)      [above right of=D2]  {$[18,17,10]$};
\node(F2)      [right of=E2]  {$\dots$};
\node(E3)      [ right of=D1]  {$[16,15,10]$};
\node(F3)      [right of=E3]  {$\dots$};
\draw(A)     edge [ above]  node {0.2526} (B1);
\draw(B1)   edge [ above]  node {0.1308} (C1);
 \draw (C1) edge [ left ]  node {0.0401}  (D1);
\draw(C1)   edge [ left]  node {0.0401} (D2);
\draw(E1)   edge [ left]  node {0.0124} (D2);
\draw(E2)     edge [left]  node {0.0124}  (D2);
\draw(E3)     edge [ above]  node {0.0145}  (D1);
\end{tikzpicture}
\caption{ A sequence of highest probability weight self-aggregation  results that starts with   $\lambda=[3,1,1]$.}
\label{Fig12}
\end{figure}
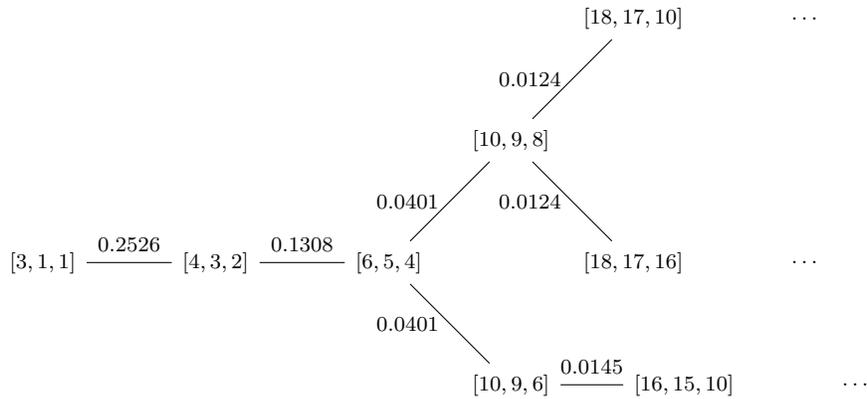

         \begin{figure}[h!]%
\centering
\begin{tikzpicture}[node distance= 2.3cm]
 \tikzstyle{every node}=[font=\tiny]
\node(A)                           {$[7,5,3]$};
\node(B1)      [right of=A]       {$[10,8,7]$ };
\node(C1)      [right  of=B1]       {$[17,15,10]$ };
\node(D1)      [ right of=C1]  {$[32,17,15]$};
\node(F1)      [ right of=D1]  {$[32,32,32]$};
\node(G3)      [ right of=F1]        {($465$ results)};
\draw(A)     edge [ above]  node {0.0546}(B1);
\draw(B1)   edge [ above]  node {0.0137}(C1);
\draw(C1)    edge [ above]  node {0.0077} (D1);
\draw(D1)     edge [ above]  node {0.0182} (F1);
\draw(F1)       edge (G3);
\end{tikzpicture}
\caption{ A sequence of highest probability weight self-aggregation  results that starts with   $\lambda=[7,5,3]$.}
\label{Fig13}
\end{figure}

\subsection {Conclusions}
We defined  two  lattice growth models, where clusters are represented by  rectangular boxes, that either move without rotations (aggregation of boxes), or can rotate  before an attachment (aggregation of partitions).  Under assumption of  equal probability of attachments   at any lattice  point of the boundary of the box, we described the set of all possible aggregations and its local statistical characteristics. In particular, we derived the  generating function of the moments of  the random variable of the number of directions of the growth of   an $l$-dimensional box. The combinatorics of these calculations is given by elementary symmetric functions. We also observe that Markov chain of aggregations  with a unit box provides a recurrence relation of  multidimensional  Delannoy numbers with non-constant weights. 

 In 2-dimensional case we  obtained a nice  evidence in support of   the mathematical  interpretation \cite{Sor1, Sor2}  of the peak value  point of the  distribution of proportions of clusters in   DLCA model. 
 We proved that   the most frequent transitions in  self-aggregations of rectangles  stabilize to side-to-end pattern, as in RHM model, and for this reason their  proportions  grow by Fibonacci-like rule  with the limit value $\varphi\simeq 1.618,$.  We observed that  transitional probability  of this sequence  tends to $\frac{1}{2}-\frac{1}{\varphi+1}\simeq 0.12$, while transitional  probabilities  of  all other results diminish to zero. 
 
 In dimensions $l\ge 3$ the approach of tracing self-aggregations  with the highest probability weight does  not provide  immediate evidence for  proposed in \cite{Sor1, Sor2}   interpretation of the peak value of proportions of aggregated clusters. Still,  the behavior of sequences of   the most frequent transitions of  self-aggregations  remains an interesting problem that we plan to investigate further.


\begin{thebibliography}{99}
   \bibitem{BJK84}
 Botet, R., Jullien, R., and Kolb, M. (1984). Hierarchical Model for Irreversible
Kinetic Cluster Formation. J. Phys. A: Math. Gen., 17:L75–L79.
 


\bibitem{Del4} 
Caughman, John S.; Dunn, Charles L.; Neudauer, Nancy Ann; Starr, Colin L.
\emph {Counting lattice chains and Delannoy paths in higher dimensions},
Discrete Math. 311 (2011), no.16, 1803–1812.

\bibitem{Del3}
Caughman, John S, Haithcock, Clifford R., Veerman, J. J. P.
\emph{A note on lattice chains and Delannoy numbers}
Discrete Math.308 (2008), no.12, 2623–2628.


\bibitem{Chung}
Chung, Kai Lai
\emph{
Markov chains with stationary transition probabilities}
Second edition
Die Grundlehren der mathematischen Wissenschaften, Band 104
Springer-Verlag New York, Inc., New York, (1967). 


\bibitem{DF90} 
Diaconis, P. Fulton, W.
\emph{A growth model, a game, an algebra, Lagrange inversion, and characteristic classes}.
Commutative algebra and algebraic geometry, II (Italian) (Turin, 1990)
Rend. Sem. Mat. Univ. Politec. Torino 49 (1991), no.1, 95–119.



\bibitem{Del1}
Grau, José María; Oller-Marcén, Antonio M.; Varona, Juan Luis
\emph {A class of weighted Delannoy numbers}, 
Filomat 36 (2022), no.17, 5985–6007.


\bibitem{Sor1} 
Heinson, W.R., Chakrabarti, A., Sorensen, C.M.
\emph{Divine proportion shape invariance of diffusion limited cluster-cluster aggregates}
(2015) Aerosol Science and Technology, 49 (9), pp. 786-792. 



  \bibitem{Mac} 
 Macdonald, I. G.
\emph{Symmetric functions and Hall polynomials.}
Oxf. Class. Texts Phys. Sci.
The Clarendon Press, Oxford University Press, New York, (2015). 




    \bibitem{M98} 
Meakiin , P. , 
\emph{Fractals, Scaling, and Growth Far From Equilibrium}
(Cambridge: Cambridge University Press), (1998).



 \bibitem{M99} 
Meakin, P.
\emph{Historical introduction to computer models for fractal aggregates}
(1999) Journal of Sol-Gel Science and Technology, 15 (2), pp. 97-117.


\bibitem{San2000} 
Sander, L.M.
\emph{Diffusion-limited aggregation: A kinetic critical phenomenon?}
(2000) Contemporary Physics, 41 (4), pp. 203-218. 

\bibitem{Sor2} 
Sorensen, C.M., Oh, C.
\emph{Divine proportion shape preservation and the fractal nature of cluster-cluster aggregates}
(1998) Physical Review E - Statistical Physics, Plasmas, Fluids, and Related Interdisciplinary Topics, 58 (6), pp. 7545-7548.



\bibitem{Del2}
Wang, Yi, Zheng, Sai-Nan, Chen, Xi
\emph{Analytic aspects of Delannoy numbers},
Discrete Math.342 (2019), no.8, 2270–2277.



  \bibitem{WB89}
Warren, P.B., Ball, R.C.
\emph{ Anisotropy and the approach to scaling in monodisperse reaction-limited cluster-cluster aggregation}
(1989) J. of Phys A: Math. and Gen., 22 (9), art. no. 027, pp. 1405-1413. 



   \bibitem{WS81} 
  Witten, T.A., Sander, L.M.
\emph{Diffusion-limited aggregation, a kinetic critical phenomenon}
(1981) Physical Review Letters, 47 (19), pp. 1400-1403. 
  
  

    \bibitem{SO86} 
\emph{On Growth and Form
Fractal and Non-Fractal Patterns in Physics}
Editors: H. Eugene Stanley, Nicole Ostrowsky
NATO Science Series E: (NSSE, volume 100),
 (1986).
 










\bibitem{Sor3*}
Heinson, W.R., Sorensen, C.M., Chakrabarti, A.
\emph{ A three parameter description of the structure of diffusion limited cluster fractal aggregates}
(2012) Journal of Colloid and Interface Science, 375 (1), pp. 65-69.






\bibitem{Sor4*}
Sorensen, C.M.
\emph{The Mobility of Fractal Aggregates: A Review}, Aerosol Science and Technology, 45:7, 765-779,
(2011). 





 





 





\end{thebibliography}
\end{document}